\documentclass[11pt]{amsart}

\usepackage{amsmath,amssymb,latexsym,amsthm,newlfont,
enumerate,enumitem}
\RequirePackage[dvipsnames,usenames]{color}

\usepackage{graphicx}
\usepackage[all]{xy}

\usepackage{hyperref}

\theoremstyle{plain}

\newtheorem{thm}{Theorem}[section]

\newtheorem{con}[thm]{Conjecture}
\newtheorem*{EffbSemCon}{Effective B-Semiampleness Conjecture}
\newtheorem*{bSemCon}{B-Semiampleness Conjecture}

\theoremstyle{definition}

\newtheorem{dfn}[thm]{Definition}

\newtheorem{rem}[thm]{Remark}
\newtheorem{ter}[thm]{Terminology}
\newtheorem{exa}[thm]{Example}

\theoremstyle{remark}

\newcommand{\N}{\mathbb{N}}
\newcommand{\C}{\mathbb{C}}
\newcommand{\R}{\mathbb{R}}
\newcommand{\Q}{\mathbb{Q}}
\newcommand{\PS}{\mathbb{P}}

\newcommand{\OO}{\mathcal{O}}

\newcommand{\can}{\mathrm{can}}

\DeclareMathOperator{\rk}{rk}

\DeclareMathOperator{\codim}{codim}

\DeclareMathOperator{\Sing}{Sing}
\DeclareMathOperator{\Exc}{Exc}

\DeclareMathOperator{\Supp}{Supp}

\DeclareMathOperator{\ddiv}{div}

\begin{document}
\title[A travel guide to the canonical bundle formula]{A travel guide to\\ the canonical bundle formula}

\author{Enrica Floris}
\address{Universit\'e de Poitiers, Laboratoire de Math\'ematiques et Applications,\linebreak UMR~CNRS 7348, T\'el\'eport 2, Boulevard Marie et Pierre Curie, BP 30179, 86962 Futuroscope Chasseneuil Cedex, France}
\email{enrica.floris@univ-poitiers.fr}

\author{Vladimir Lazi\'c}
\address{Fachrichtung Mathematik, Campus, Geb\"aude E2.4, Universit\"at des Saarlandes, 66123 Saarbr\"ucken, Germany}
\email{lazic@math.uni-sb.de}

\thanks{
Lazi\'c was supported by the DFG-Emmy-Noether-Nachwuchsgruppe ``Gute Strukturen in der h\"oherdimensionalen birationalen Geometrie". During the initial stage of this project Floris was funded by the Max Planck Institute for Mathematics in Bonn. We would like to thank A.\ Golota for pointing out the reference \cite{FF17} and for useful conversations.
}

\begin{abstract}
We survey known results on the canonical bundle formula and its applications in algebraic geometry.
\end{abstract}

\maketitle
\setcounter{tocdepth}{1}
\tableofcontents

\section{Introduction}

The Minimal Model Program (MMP) predicts that every projective pair with mild singularities is birationally built out of three classes of pairs: those whose log canonical classes are ample, numerically trivial or anti-ample.

More precisely, let $(X,\Delta)$ be a log canonical pair. Then there should exist a birational contraction $\varphi\colon (X,\Delta)\dashrightarrow (X_{\min},\Delta_{\min})$ together with a fibration $f\colon (X_{\min},\Delta_{\min})\to X_\can$ so that $K_{X_{\min}}+\Delta_{\min}\sim_\Q f^*A$, for a suitable ample $\Q$-divisor $A$ on $X_\can$. Note that the Iitaka dimension of $K_X+\Delta$ restricted to a general fibre of the composed map $f\circ\varphi$ is zero.

It is a natural and important question to determine whether singularities of the MMP are preserved in this process. The singularities of $(X_{\min},\Delta_{\min})$ are the same as those of $(X,\Delta)$. On the other hand, it remains an open problem whether there exists a boundary divisor $\Delta_\can$ on $X_\can$ such that $(X_\can,\Delta_\can)$ is log canonical and $K_{X_{\min}}+\Delta_{\min}\sim_\Q f^*(X_\can+\Delta_\can)$; in other words, whether singularities of $(X,\Delta)$ \emph{descend} to the canonical model $X_\can$.

When the singularities of $(X,\Delta)$ are klt, it is known by a work of Ambro and Kawamata that such a divisor exists; this result has already had numerous consequences in birational geometry. However, the proof is not constructive: more precisely, one loses control of the coefficients of $\Delta$ at the last step. It is desirable that the singularities of $(X_\can,\Delta_\can)$ reflect in a canonical way the singularities of $(X,\Delta)$.

In general, with notation as above, it is known that 
$$A\sim_\Q K_{X_\can}+B_{X_\can}+M_{X_\can},$$ 
where $B_{X_\can}$ -- the \emph{discriminant} -- is closely related to the singularities of $f$, and the divisor $M_{X_\can}$ -- the \emph{moduli divisor} -- conjecturally carries information on the birational variation of the fibres of $f$. A formula of this form is called the \emph{canonical bundle formula}. 

This paper is an attempt to give an account of all the known results on the canonical bundle formula, and to serve as a guide to those wishing to study this important subject.

\section{Lc-trivial fibrations}

We work over $\C$. We denote by $\equiv$, $\sim$ and $\sim_\Q$ the numerical, linear and $\Q$-linear equivalence of divisors respectively. 

For a Weil $\Q$-divisor $D=\sum d_i D_i$, for a real number $r$ we denote  $D^{\leq r}:=\sum_{d_i\leq r}d_i D_i$. If $f\colon X\to Y$ is a proper surjective morphism between normal varieties and $D$ is a Weil $\R$-divisor on $X$, then $D_v$ and $D_h$ denote the vertical and the horizontal part of $D$ with respect to $f$. In this setup, we say that $D$ is \emph{$f$-exceptional} if $\codim_Y\Supp f(D)\geq2$.

In this section we introduce the main topic of this survey -- \emph{lc-trivial fibrations}. In this section we define them and give some examples which will accompany us through the paper.

We first need to introduce singularities of pairs. This is by now a standard topic in higher dimensional birational geometry, and good references are \cite{KM98} and \cite{Kol13}.

A \emph{pair} $(X,\Delta)$ consists of a normal variety $X$ and a Weil $\Q$-divisor $\Delta$ such that $K_X+\Delta$ is $\Q$-Cartier. A pair $(X,\Delta)$ is \emph{log smooth} if $X$ is smooth and the support of $\Delta$ is a simple normal crossings divisor.

A \emph{log resolution} of a pair $(X,\Delta)$ is a birational morphism $f\colon Y\to X$ such that the exceptional locus $\Exc(f)$ is a divisor and the pair $\big(Y,\Supp(f_*^{-1}\Delta+\Exc(f)\big)$ is log smooth.

If $(X,\Delta)$ be a pair and if $\pi\colon Y\to X$ is a birational morphism with $Y$ is normal, we can write
$$K_Y\sim_\Q \pi^*(K_X+\Delta)+\sum a(E_i,X,\Delta)\cdot E_i,$$
where $E_i\subseteq Y$ are distinct prime divisors and the numbers $a(E_i,X,\Delta)\in\Q$ are called \emph{discrepancies}. The order of vanishing at the generic point of each $E_i$ defines a \emph{geometric valuation} on $\C(X)$. The pair $(X,\Delta)$ is \emph{klt}, respectively \emph{log canonical}, if $a(E,X,\Delta)>-1$, respectively $a(E,X,\Delta) \geq -1$, for every geometric valuation $E$ over $X$. 

Much of what we say in this paper can be generalised to pairs $(X,\Delta)$, where $\Delta$ is allowed to have real coefficients. We stick to rational divisors mostly for reasons of clarity and simplicity.

\subsection{Definition and first examples}

The objects for which we can write a canonical bundle formula are called \emph{lc-trivial fibrations}. 

\begin{dfn}\label{dfn:lctrivial}
Let $(X,\Delta)$ be a pair. A morphism $f \colon (X,\Delta) \rightarrow Y$ to a normal projective variety $Y$ is a \emph{klt-trivial}, respectively \emph{lc-trivial}, fibration if:
\begin{enumerate}
\item[(a)] $f$ is a surjective morphism with connected fibres,
\item[(b)] $(X,\Delta)$ has klt, respectively log canonical, singularities over the ge\-ne\-ric point of $Y$,
\item[(c)] there exists a $\Q$-Cartier $\Q$-divisor $D$ on $Y$ such that
$$K_X+\Delta\sim_\Q f^*D,$$
\item[(d)] there exists a log resolution $\pi'\colon X'\rightarrow X$ of $(X,\Delta)$ such that, if $\mathcal E$ is the set of all geometric valuations over $X$ which are defined by a prime divisor $E$ on $X'$ such that $a(E,X,\Delta)>-1$, and if we denote $\Xi'=\sum\limits_{E\in\mathcal E} a(E,X,\Delta)\cdot E$, then
$$\rk (f\circ\pi')_*\OO_{X'}(\lceil \Xi'\rceil) = 1.$$
\end{enumerate}  
\end{dfn}

\begin{ter}
In \cite{Amb04}, klt-trivial fibrations as in Definition \ref{dfn:lctrivial} are called lc-trivial fibrations.
\end{ter}

\begin{rem}
We make a few comments on the condition (d) in Definition \ref{dfn:lctrivial}. For simplicity, assume that the pair $(X,\Delta)$ is klt. Note that then
$$\Xi'\sim_\Q K_{X'}-\pi'^*(K_X+\Delta).$$
The divisor $\lceil \Xi'\rceil$ is effective on the generic fibre of $f\circ\pi'$ by (b), hence $\rk (f\circ\pi')_*\OO_{X'}(\lceil \Xi'\rceil) \geq 1$, hence the point of (d) is the opposite inequality. 

The most important case to keep in mind is that when the divisor $\Delta$ is effective on the generic fibre: indeed, in that case the divisor $\lceil \Xi'\rceil$ is an effective exceptional divisor on the generic fibre of $f\circ\pi'$, and the condition (d) is immediate. However, in order to be able to study lc-trivial fibrations by applying basic operations of birational geometry in \S\ref{subsec:basechange}, it is crucial to allow divisors $\Delta$ with negative coefficients.

One more thing to notice is that if (d) holds for a log resolution $\pi'$, then it holds on any log resolution $\pi''\colon X''\to X$ which factors through $\pi'$. Define the divisor $\Xi''$ on $X''$ analogously as in Definition \ref{dfn:lctrivial}, and let $\theta\colon X''\to X'$ be the induced morphism. Since $X'$ and $X''$ are smooth, there exists an integral effective divisor $E$ such that $K_{X''}\sim \theta^*K_{X'}+E$. Thus,
\begin{align*}
f_*\pi''_*\OO_{X''}(\lceil \Xi''\rceil)&=f_*\pi''_*\OO_{X''}(\lceil \theta^*\Xi'+E\rceil)=f_*\pi''_*\OO_{X''}(\lceil \theta^*\Xi'\rceil+E)\\
&\subseteq f_*\pi''_*\OO_{X''}(\theta^*\lceil \Xi'\rceil+E)=f_*\pi'_*\OO_{X'}(\lceil \Xi'\rceil),
\end{align*}
where we used that $\lceil \theta^*\Xi'\rceil\leq \theta^*\lceil \Xi'\rceil$. Therefore, $\rk (f\circ\pi'')_*\OO_{X''}(\lceil \Xi''\rceil)=1$.

In general, one can show that (d) is independent of the choice of the resolution by using \cite[Lemma 2.7]{Fuj10}.
\end{rem}

Now we can formulate the canonical bundle formula associated to an lc-trivial fibration.

\begin{dfn}\label{dfn:cbf}
Let $f\colon (X,\Delta)\to Y$ be an lc-trivial fibration such that $K_X+\Delta\sim_\Q f^*D$ for some $\Q$-Cartier $\Q$-divisor $D$ on $Y$. If $P\subseteq Y$ is a prime divisor, the \emph{log canonical threshold} of $f^*P$ with respect to $(X,\Delta)$ is
$$\gamma_P=\sup\{t\in\R\mid (X,\Delta+tf^*P) \textrm{ is log canonical over the generic point of } P\}.$$
The condition that $(X,\Delta+tf^*P)$ is \emph{log canonical over the generic point of $P$} means that for every geometric valuation $E$ over $X$ which surjects onto $P$, we have $a(E,X,\Delta+tf^*P)\geq-1$. 
The \emph{discriminant} of $f$ is
$$\textstyle B_Y=\sum_P(1-\gamma_P)P.$$
Fix $\varphi\in\C(X)$ and the smallest positive integer $r$ such that $K_X + \Delta +\frac{1}{r}\ddiv\varphi = f^*D$. Then there exists a unique Weil $\Q$-divisor $M_Y$, the \emph{moduli part} of $f$, such that 
$$\textstyle K_X + \Delta +\frac{1}{r}\ddiv\varphi = f^*(K_Y+B_Y+M_Y).$$
This formula is the \emph{canonical bundle formula} associated to $f$.
\end{dfn}

\begin{rem}
The definition of the discriminant as above first appeared in \cite[Theorem 2]{Kaw98}. The discriminant is a Weil $\Q$-divisor on $Y$, and it is effective if $\Delta$ is effective. We notice that the discriminant is defined as an actual divisor, while the moduli part is defined only up to $\Q$-linear equivalence: it depends on the choice of $D$. Further, if $G$ is a $\Q$-divisor on $Y$, then $f\colon (X,\Delta+f^*G)\to Y$ is an lc-trivial fibration with discriminant $B_Y+G$ and moduli divisor $M_Y$.

Note that if we are interested in proving properties of the moduli divisor of an lc-trivial fibrations as above, we may always assume that the pair $(X,\Delta)$ is log canonical by \cite[Remark 3.6]{FL18}, although we may not assume that $\Delta$ is effective.
\end{rem}

\begin{exa}
Assume that $X$ is smooth, that $\Delta=0$, that $Y$ is a curve, that $f^{-1}(P)$ is smooth and that $f^*P=mf^{-1}(P)$ is a multiple fibre. Then $\gamma_P=\frac{1}{m}$.
\end{exa}

\begin{exa}\label{ellfibr}
This example is historically the first example of a canonical bundle formula. Let $f\colon X\to C$ be an elliptic fibration, that is, $X$ is a smooth surface, $C$ is a smooth curve and a general fibre of $f$ is a smooth elliptic curve. We assume furthermore that $f$ is relatively minimal: there are no $(-1)$-curves contained in the fibres of $f$. Kodaira in \cite[Theorem 6.2]{Ko60} classified the singular fibres of $f$. Kodaira's canonical bundle formula reads as
$$K_X\sim f^*(K_C+B_C+M_C),$$
where $B_C$ is defined in terms of the classification of the singular fibres and by \cite{Ko60, Ue73} we have $12M_C=j^*\mathcal{O}_{\PS^1}(1)$, with $j\colon C\to \PS^1$ being the $j$-invariant. For a detailed account on Kodaira's canonical bundle formula see \cite[Chapter V, \S 7--\S13]{BPV}.
\end{exa}

\begin{exa}
Let $X=\PS^1\times\PS^1$ and let $D$ be a reduced divisor of bidegree $(d,k)$ with $d\geq 2$. Let $\Delta=\frac{2}{d}D$ and let $f\colon (X,\Delta)\to \PS^1$ be the projection onto the second factor. Then $f$ is an lc-trivial fibration. Indeed, $K_X+\Delta$ has bidegree $(-2,-2)+\frac{2}{d}(d,k)=(0,-2+\frac{2k}{d})$ and therefore is the pullback of a divisor from $\PS^1$.
\end{exa}


\subsection{Base change property}\label{subsec:basechange}

In this subsection we investigate how canonical bundle formulas behave under base change. This will help improve the properties of the moduli part of a canonical bundle formula, at least on a sufficiently high birational model.

If $f \colon(X,\Delta)\to Y$ is a klt-trivial fibration (respectively lc-trivial), and if we consider a base change diagram
\begin{equation}\label{eq:diag}
\begin{gathered}
\xymatrix{
(X',\Delta') \ar[r]^\tau \ar[d]_{f'} & (X,\Delta)\ar[d]^{f}\\
Y' \ar[r]_{\rho}&Y,
}
\end{gathered}
\end{equation}
where $\rho$ is a proper generically finite morphism, $X'$ is the normalisation of the fibre product and $\Delta'$ is defined so that we have
$$K_{X'}+\Delta'=\tau^*(K_X+\Delta),$$
then $f' \colon(X',\Delta')\to Y'$ is also a klt-trivial (respectively lc-trivial) fibration. In the rest of the paper, we implicitly refer to this klt-trivial fibration when writing $M_{Y'}$ and $B_{Y'}$ for the moduli part and discriminant.

If $\rho$ is birational, then $\rho_*M_{Y'}=M_Y$ and $\rho_*B_{Y'}=B_Y$; in other words, these collections of divisors form \emph{b-divisors} see for instance \cite[Section 1.2]{Amb04}. 

The following is the \emph{base change property} of canonical bundle formulas.

\begin{thm}\label{nefness}
Let $f \colon(X,\Delta)\to Y$ be an lc-trivial fibration. Then there exists a proper birational morphism $Y'\to Y$ such that for every proper birational morphism $\pi \colon Y''\to Y'$ we have:\begin{enumerate}
\item[(i)] $K_{Y'}+B_{Y'}$ is a $\Q$-Cartier divisor and $K_{Y''}+B_{Y''}=\pi^*(K_{Y'}+B_{Y'})$,
\item[(ii)] $M_{Y'}$ is a nef $\Q$-Cartier divisor and $M_{Y''}=\pi^*M_{Y'}$.
\end{enumerate}
\end{thm}

The first version of Theorem \ref{nefness} is \cite[Theorem 2]{Kaw98}, which essentially shows the nefness of the moduli part; this is also the point of the proof where the condition (d) in Definition \ref{dfn:lctrivial} is used. Theorem \ref{nefness} has been proved in this form for klt-trivial fibrations by Ambro \cite[Theorem 0.2]{Amb04}. For lc-trivial fibrations, it was proved by Koll\'ar \cite[Theorem 8.3.7]{Kol07} and  \cite[Theorem 3.6]{FG14}, with an alternative proof in \cite{Flo14a}. 

\medskip

In the context of the previous theorem, we say that $M_{Y'}$ \emph{descends} to $Y'$, and we call any such $Y'$, where additionally $B_{Y'}$ has simple normal crossings support, an \emph{Ambro model} for $f$. 

\medskip

We give a brief sketch of the proof of the nefness of the moduli divisor in the previous theorem; very good references are \cite[Lemma 5.2]{Amb04} and especially \cite[Theorem 8.5.1 and \S8.10]{Kol07}, where many more details are given. By the proof of \cite[Theorem 8.5.1]{Kol07}, it suffices to show the claim after making a suitable generically finite base change and taking the cyclic cover of $X$ associated to $\sqrt[r]{\varphi}$. The base change as in \eqref{eq:diag} that we are aiming for is a composition of a log resolution with a Kawamata cover such that, on an open subset of $U'\subseteq Y'$ whose complement has codimension at least $2$ in $Y'$, the local systems $R^if'_*\C_{X'}|_{U'}$ have unipotent monodromies; this is the content of \cite[ 8.10.7--8.10.10]{Kol07}. We may also assume that $M_{Y'}$ is a Cartier divisor. If $f$ is a klt-trivial fibration, then by \cite[Theorem 8.5.1]{Kol07} the line bundle $\OO_{Y'}(M_{Y'})$ is a quotient of the locally free sheaf $f'_*\omega_{X'/Y'}$. Then one applies \cite[Theorem 5]{Kaw81}, which asserts that $f'_*\omega_{X'/Y'}$ is the canonical extension of the bottom piece of the Hodge filtration of $R^{\dim X-\dim Y}f'_*\C_{X'}|_{U'}$, and hence its quotients are nef by the same result. A similar statement holds also for lc-trivial fibrations. 

Recently, a stronger statement was shown in \cite[Theorem 1.1]{FF17}. The result shows that \cite[Theorem 5]{Kaw81} can be improved to show that not only any quotient of $f'_*\omega_{X'/Y'}$ is nef, but moreover, it carries a singular metric whose Lelong numbers are all zero. This is much stronger than being nef, as it implies, in particular, that the multiplier ideal associated to this metric is trivial. 

We summarise this in the following result.

\begin{thm}
Let $f \colon(X,\Delta)\to Y$ be a klt-trivial fibration. Then there exists a proper birational morphism $Y'\to Y$ such that for every proper birational morphism $\pi \colon Y''\to Y'$ we have:
\begin{enumerate}
\item[(i)] $K_{Y'}+B_{Y'}$ is a $\Q$-Cartier divisor and $K_{Y''}+B_{Y''}=\pi^*(K_{Y'}+B_{Y'})$,
\item[(ii)] $M_{Y'}$ is a $\Q$-Cartier divisor carrying a singular metric whose all Lelong numbers are zero, and $M_{Y''}=\pi^*M_{Y'}$.
\end{enumerate}
\end{thm}

The proof of part (ii) of the theorem is the same as the sketch of the proof of Theorem \ref{nefness}(ii) above, since a $\Q$-divisor carries a singular metric whose all Lelong numbers are zero if and only if its pullback by a proper surjective map carries a singular metric whose all Lelong numbers are zero by \cite[Corollary 4]{Fav99}.

\subsection{Inversion of adjunction}

In order to appreciate the following result and to see why base change property is important, let us go back to the construction of the canonical bundle formula. Recall that the discriminant divisor was constructed in terms of \emph{local} log canonical thresholds, that is, log canonical thresholds over the generic point of a prime divisor; in particular, with notation from Definition \ref{dfn:cbf}, for some prime divisor $P$ on $Y$, the pair $(X,\Delta+\gamma_P f^*P)$ does not have to be \emph{globally} log canonical. However, the following \emph{inversion of adjunction} \cite[Theorem 3.1]{Amb04} states that this is precisely what happens on an Ambro model.

\begin{thm}\label{thm:invAdjunction}
Let $f\colon(X,\Delta)\rightarrow Y$ be an lc-trivial fibration, and assume that $Y$ is an Ambro model for $f$. Then $(Y,B_Y)$ has klt, respectively log canonical, singularities in a neighbourhood of a point $y\in Y$ if and only if $(X,\Delta)$ has klt, respectively log canonical, singularities in a neighbourhood of $f^{-1}(y)$.
\end{thm}

Note that Theorem \ref{thm:invAdjunction} is stated for klt-trivial fibrations in \cite{Amb04}, but the proof extends verbatim to the lc-trivial case by using Theorem \ref{nefness}.

We finish this subsection with the following nice result \cite[Proposition 3.1]{Amb05a} which we will apply in the proof of Theorem \ref{thm:kltdescent} below.

\begin{thm}\label{thm:pullbackAmbro}
Let $f\colon(X,\Delta)\rightarrow Y$ be a klt-trivial fibration, and assume that $Y$ is an Ambro model for $f$. Then for every base change by a surjective morphism $w\colon W\to Y$ we have $K_W+B_W\sim_\Q w^*(K_Y+B_Y)$ and $M_W\sim_\Q w^*M_Y$.
\end{thm}

\subsection{Coefficients of the moduli divisor}

Often in applications one needs to bound the denominators of $M_Y$. For instance, such bounds were used in \cite[Theorem 1.2]{Flo14}.

\begin{thm}
For each nonnegative integer $b$ there exists an integer $N$ depending on $b$ such that the following holds.

Let $f\colon(X,\Delta)\to Y$ be a klt-trivial fibration with a general fibre $F$, and let $r$ be the smallest positive integer such that $r(K_F+\Delta|_F)\sim0$. Let $E\to F$ be the associated $r$-th cyclic cover and let $\overline{E}$ be a resolution of singularities of $E$. If $\dim H^{\dim\overline{E}}(\overline{E},\C)=b$, then the divisor $NM_Y$ is integral.
\end{thm}

The result was proved in \cite[Theorem 3.1]{FM00} when $\Delta=0$, but the same proof works for klt-trivial fibrations \cite[Theorem 5.1]{Flo14}.

A more refined result holds when a general fibre is a rational curve, \cite[Theorem 1.6(2)]{Flo13}.

\begin{thm}
Fix a positive integer $r$. For a prime number $q$ set $s(q) = \max\{s \mid q^s \leq 2r\}$ and define
$$N=\prod_{q \text{ prime}}q^{s(q)}.$$
\begin{enumerate}
\item[(a)] Let $f\colon(X,\Delta)\to Y$ be an lc-trivial fibration whose general fibre $F$ is a rational curve, and assume that $r$ is the smallest positive integer such that $r(K_F+\Delta|_F)\sim0$. Then the divisor $NM_Y$ is integral. 
\item[(b)] Assume $r$ is odd. Then there exists an lc-trivial fibration $f\colon(X,\Delta)\to Y$ such that if $v$ is the smallest integer for which the divisor $v M_Y$ is integral, then $v = N/r$.
\end{enumerate}
\end{thm}

\subsection{Goodness of moduli divisors}

Now we come to the central topic of this survey, already announced in the introduction: the descent of singularities. Since we already know the nefness of the moduli divisor by Theorem \ref{nefness}, if it were additionally big, then this would allow to conclude in many cases. Bigness is too much to ask; however, the following result of Ambro \cite[Theorem 3.3 and Proposition 4.4]{Amb05a} turn out to be almost as good.

\begin{thm}\label{ambro1}
Let $f\colon(X,\Delta)\to Y$ be a klt-trivial fibration between normal projective varieties such that $\Delta$ is effective over the generic point of $Y$. Then there exists a diagram
$$
\xymatrix{
(X,\Delta)\ar[d]_f && (X^+,\Delta^+)\ar[d]^{f^+}\\
Y & W \ar[l]^{\tau}\ar[r]_{\rho}&Y^+
}
$$
such that:
\begin{enumerate}
\item[(i)] $f^+\colon(X^+,\Delta^+)\rightarrow Y^+$ is a klt-trivial fibration,
\item[(ii)] $\tau$ is generically finite and surjective, and $\rho$ is surjective,
\item[(iii)] if $M_Y$ and $M_{Y^+}$ are the moduli divisors of $f$ and $f^+$ respectively, then $M_{Y^+}$ is big and, after possibly a birational base change, we have $\tau^*M_Y=\rho^*M_{Y^+}$,
\item[(iv)] there exists a non-empty open set $U\subseteq W$ and an isomorphism
$$
\xymatrix{
(X,\Delta)\times_{Y} U\ar[rd]\ar[rr]^{\simeq\quad}&&(X^+,\Delta^+)\times_{Y^+} U\ar[ld]\\
&U,&
}
$$
\item[(v)] if there exists an isomorphism
$$\Phi\colon (X,\Delta)\times_Y U\to (F, \Delta_F)\times U$$
over a non-empty open subset $U\subseteq Y$, then $\Phi$ extends to an isomorphism over 
$$Y^0=Y\setminus \big(\Supp B_Y\cup \Sing (Y)\cup f(\Supp \Delta_v^{\leq0})\big).$$
\end{enumerate}
\end{thm}


Note that in (v) one does not need that $\Delta$ is effective on the generic fibre of $f$.

For us, the most important part of this result is (iii). Its immediate consequence is the descent of klt singularities \cite[Theorem 0.2]{Amb05a}.

\begin{thm}\label{thm:kltdescent}
Let $(X,\Delta)$ be a projective klt pair with $\Delta$ effective, and let $f\colon X\to Y$ be a surjective morphism to a normal projective variety such that $K_X+\Delta\sim_\Q f^*D$ for some $\Q$-Cartier $\Q$-divisor $D$ on $Y$. Then there exists an effective $\Q$-divisor $\Delta_Y$ on $Y$ such that the pair $(Y,\Delta_Y)$ is klt and 
$$K_X+\Delta\sim_\Q f^*(K_Y+\Delta_Y).$$
\end{thm}

\begin{proof}
We use the notation from Theorem \ref{ambro1}, which clearly applies in our situation. We have the base change diagram
$$
\xymatrix{
(W,\Delta_W) \ar[r]^{w} \ar[d]_{f_W} & (X,\Delta)\ar[d]^{f}\\
W \ar[r]_{\tau}&Y.
}
$$
We may additionally assume that $\tau$ factors through an Ambro model $\pi\colon Y'\to Y$ of $f$, and denote by $\sigma\colon W\to Y'$ the induced morphism. By replacing $W$ by a suitable log resolution, by an easy argument with the Stein factorisation of $\sigma$ together with Theorem \ref{thm:pullbackAmbro} we may assume that $W$ is an Ambro model of $f_W$. 

Now, $\tau^*D\sim_\Q K_W+B_W+M_W$, and by the inversion of adjunction, Theorem \ref{thm:invAdjunction}, the pair $(W,B_W)$ is klt. Since the divisor $M_{Y^+}$ is nef and big, by using a version of Kodaira's trick \cite[Proposition 2.61]{KM98} together with Bertini's theorem, we may find an effective $\Q$-divisor $E_W$ on $W$ such that $M_W\sim_\Q E_W$ and such that the pair $(W,B_W+E_W)$ is klt. 

By Theorem \ref{thm:pullbackAmbro} we have $K_W+B_W\sim_\Q \sigma^*(K_{Y'}+B_{Y'})$ and $M_W\sim_\Q \sigma^*M_{Y'}$. Hence, if we denote $E_{Y'}=\frac{1}{\deg\sigma}\sigma_*E$, we have $M_{Y'}\sim_\Q E_{Y'}$ and
$$K_W+B_W+E\sim_\Q \sigma^*(K_{Y'}+B_{Y'}+E_{Y'}).$$
Then the pair $(Y',B_{Y'}+E_{Y'})$ is klt by \cite[Proposition 5.20]{KM98}. Setting 
$$\Delta_Y=\pi_*(B_{Y'}+E_{Y'})=B_Y+\pi_*E,$$
we have $K_Y+\Delta_Y\sim_\Q D$, and $(Y,\Delta_Y)$ is a klt pair. Since $\Delta$ is effective, the divisor $B_Y$ is effective, thus $\Delta_Y$ is effective.
\end{proof}

Finally, we mention that Theorem \ref{ambro1}(i)-(iii) was generalised to lc-trivial fibrations with $\Delta\geq0$ and $(X,\Delta)$ is log canonical in \cite[Lemma 1.1]{FG14}. The proof involves running a careful MMP in order to reduce to a situation where one has a klt-trivial fibration.


\begin{exa}
Theorem \ref{ambro1}(iv) does not hold for lc-trivial fibrations. For instance, let $X=\PS^1\times\PS^1$, let $f$ be the second projection, let $\delta$ be the diagonal and set $\Delta=\delta+\frac{1}{2}\{0\}\times\PS^1+\frac{1}{2}\{\infty\}\times\PS^1$. Then $f\colon(X,\Delta)\to \PS^1$ is an lc-trivial fibration with discriminant supported on $\{0\}\cup\{\infty\}$. By considering log resolutions, one calculates that the discriminant is equal to $\frac{1}{2}0+\frac{1}{2}\infty$, hence the moduli divisor is torsion. Indeed, we have
$$\textstyle K_X+\Delta\sim_{\Q}f^*\big(K_{\PS^1}+\frac{1}{2}0+\frac{1}{2}\infty+M_{\PS^1}\big),$$
but the divisor $K_X+\Delta$ has bi-degree $(0,-1)$.
 
However, $f$ is not birational to a product. Indeed, it induces a family of rational curves with 4 marked points parametrised by $\PS^1$: for $t\in \PS^1$, set $p_1(t)=0$, $p_2(t)=1$, $p_3(t)=p_4(t)=t$. This family of marked curves is not trivial, therefore the fibration cannot be locally a product.
\end{exa}

\section{B-Semiampleness conjectures}

We saw above in Theorem \ref{thm:kltdescent} that klt singularities descend along a klt-trivial fibration. However, even if one had the full analogue of Theorem \ref{ambro1} in the log canonical setting one could not follow the same strategy to show that log canonical singularities descend. Moreover, one sees that the use of Kodaira's trick in the proof of Theorem \ref{thm:kltdescent} obliterated the connection of the coefficients of the divisor $\Delta_Y$ to that of the divisor $\Delta$. In order to remedy the situation, something more is needed.

The main conjecture on the canonical bundle formula predicts something much stronger: that the moduli part is \emph{semiample} on an Ambro model of the fibration. We discuss it in this section.

There are two versions of the conjecture; the stronger one was proposed in \cite[Conjecture 7.13.3]{PS09}. We start with the stronger, \emph{effective} version.

\begin{EffbSemCon}\label{effbsemi}
Fix positive integers $d$ and $r$. Then there exists an integer  $m$ depending only on $d$ and $r$, such that for any lc-trivial fibration $f\colon (X,B)\rightarrow Y$ with the generic fibre $F$, if  $\dim F=d$ and $r$ is the smallest positive integer such that $r(K_F+B|_F)\sim 0$, there exists an Ambro model $Y'$ of $f$ such that $mM_{Y'}$ is base point free. 
\end{EffbSemCon}

More generally, any conjecture as above, in which $m$ depends on some invariants of the generic fibre of the fibration, goes under the name of \emph{effective b-semiampleness}. 

In the original statement \cite[Conjecture 7.13.3]{PS09}, the constant $m$ depended on $\dim X$ and $r$. The main result of \cite{Flo14} is that it suffices to show the Effective B-Semiampleness Conjecture in the case where $Y$ is a curve.
This led to the formulation of the Effective B-Semiampleness Conjecture above.

The conjecture is widely open. We list below the cases where it is known. They all make use of some notion of moduli space for the fibres.

\begin{thm}\label{thm:eff}
The Effective B-Semiampleness Conjecture holds in the following cases:
\begin{enumerate}
\item if general fibres are elliptic curves by \cite{Ko60, Ue73}; we have $m=12$;
\item if general fibres are rational curves \cite[Theorem 8.1]{PS09}; 
\item if general fibres are K3 surfaces or abelian varieties of dimension $d$ by \cite[Theorem 1.2]{Fuj03}; then we  have $m=19k$, respectively $m=k(d+1)$, where $k$ is a weight associated to the Baily-Borel-Satake compactification of the period domain.
\end{enumerate}
\end{thm}

In the weaker version of the conjecture we lose control on the constant $m$.

\begin{bSemCon}\label{bsemi}
Let $f\colon (X,B)\rightarrow Y$ be an lc-trivial fibration. Then there exists an Ambro model $Y'$ of $f$ such that $M_{Y'}$ is semiample. 
\end{bSemCon}

More is known about this conjecture than about its effective version above, although the progress has been limited to the cases where either the bases $Y$ or the fibres of $f$ are low dimensional. We summarise the situation for low dimensional fibres in the following result.

\begin{thm}
Apart from the cases listed in Theorem \ref{thm:eff}, the B-Semi\-am\-ple\-ness Conjecture holds in the following cases:
\begin{enumerate}
\item if the fibres are surfaces of Kodaira dimension 0 by \cite[Lemma 4.1 and Corollary 6.4]{Fuj03} and \cite[Part I, (5.15.9)(ii)]{Mo87};
\item if $f$ is a klt-trivial fibration and the generic fibre is a uniruled surface not isomorphic to $\PS^2$ by \cite[Theorem 1.7]{Fil18}.
 \end{enumerate} 
 \end{thm}

The B-Semiampleness Conjecture holds for another important family of fibrations:

\begin{thm}\label{thm:torsionAmbro}
Let $f\colon(X,\Delta)\to Y$ be an lc-trivial fibration, and assume that the moduli part $M_Y$ descends to $Y$. If $M_Y\equiv0$, then $M_Y\sim_\Q 0$.

In particular, if $\dim Y=1$, then $M_Y$ is semiample.
\end{thm}

Theorem \ref{thm:torsionAmbro} is  \cite[Theorem 3.5]{Amb05a} for klt-trivial fibrations and \cite[Theorem 1.3]{Flo14} for lc-trivial fibration. Theorem 1.2 in \cite{Flo14} states that Effective B-Semiampleness Conjecture is true for klt-trivial fibrations with numerically trivial moduli part.

Another partial result is \cite[Theorem 3.2]{BC16}. They prove that a small perturbation of the moduli part in a specific direction is semiample, under some effectivity hypotheses for $K_Y+B_Y$.

\subsection{Restrictions to divisors}

As we saw in the previous subsection, the progress on the B-Semiampleness Conjecture has been concentrated on the cases of either the low dimension of the base $Y$ or the low dimension of the generic fibre of the fibration $f$. 

In \cite[Theorem A]{FL18} we obtained the following result towards the conjecture valid in every dimension. Note that the phrase \emph{the B-semiampleness Conjecture in dimension $n$} means that we consider the conjecture in the case when the dimension of the base $Y$ is $n$.

\begin{thm}\label{thm:main}
Assume the B-Semiampleness Conjecture in dimension $n-1$.

Let $(X,\Delta)$ be a log canonical pair and let $f\colon (X,\Delta)\to Y$ be an lc-trivial fibration to an $n$-dimensional variety $Y$, where the divisor $\Delta$ is effective over the generic point of $Y$. Assume that $Y$ is an Ambro model for $f$. 

Then for every birational model $\pi\colon Y'\to Y$ and for every prime divisor $T$ on $Y'$ with the normalisation $T^\nu$ and the induced morphism $\nu\colon T^\nu\to Y'$, the divisor $\nu^*\pi^*M_Y$ is semiample on $T^\nu$.
\end{thm}

As a corollary, combining with Theorem \ref{thm:torsionAmbro}, we obtain that the restriction of the moduli part to every prime divisor on every sufficiently high birational model of $Y$ is semiample if $Y$ is a surface.

We comment on the proof of Theorem \ref{thm:main}, as it will be useful in the following subsection. We first apply a base change to $Y$ and modify $(X,\Delta)$ by blowing up suitably, but we try to remember $(X,\Delta)$ along the proof. We then run a suitable relative MMP over $Y$, which contracts many ``bad'' components of $\Delta$ (in particular, those with negative coefficients); as a result, we obtain a new lc-trivial fibration $g\colon (W,\Delta_W)\to Y$ with $\Delta_W\geq0$ and with the same moduli divisor $M_Y$. Choosing a minimal log canonical centre $S$ of $(W,\Delta_W)$ which surjects onto $T$, we obtain an induced \emph{klt}-trivial fibration $g|_S\colon(S,\Delta_S)\to T'$, where $T'$ is obtained from the Stein factorisation of the morphism $S\to T$. Then we first show that $M_Y|_{T'}$ is \emph{almost} $M_{T'}$. Even though at this step we may not deduce equality between these two divisors, we can control their difference in a very precise manner. After a suitable further blowup of $Y$, we can force this difference to disappear and we conclude by induction on the dimension.

\subsection{Reduction result}

As we mentioned above, in the setup of lc-trivial fibrations $f\colon(X,\Delta)\to Y$ one does not assume that $\Delta$ is effective. Furthermore, often it is much more difficult to work with lc-trivial fibrations than with klt-trivial fibrations.

In \cite{FL18} the B-Semiampleness Conjecture is reduced to the following conjecture with much weaker hypotheses.

\begin{con}\label{con:weakbsemi}
Let $(X,\Delta)$ be a log canonical pair and let $f\colon (X,\Delta)\to Y$ be a klt-trivial fibration to an $n$-dimensional variety $Y$. If $Y$ is an Ambro model of $f$ and if the moduli divisor $M_Y$ is big, then $M_Y$ is semiample.
\end{con}

We show in \cite[Theorem E]{FL18}:

\begin{thm}\label{KLTimplLC}
Assume Conjecture \ref{con:weakbsemi} in dimensions at most $n$. Then the B-Semiampleness Conjecture holds in dimension $n$.
\end{thm}

The proof is similar to the proof of Theorem \ref{thm:main} sketched above.

\subsection{Generalisation}

In the papers \cite{Fuj18,FFL18} the authors consider slc-trivial fibrations. Those are completely analogous to lc-trivial fibrations, the difference being that the ambient space $X$ is not irreducible, but the pair $(X,\Delta)$ is \emph{slc} on the generic fibre of the fibration; such a setup appears occasionally in inductive proofs. The precise statement is \cite[Definition 4.1]{Fuj18}.
%
%
%

Then one can define, as in the case of lc-trivial fibrations, the moduli divisor and the discriminant. Then \cite[Theorem 1.2]{Fuj18} proves the analogue of Theorem \ref{nefness} for slc-trivial fibrations, and \cite[Theorem 1.3]{FFL18} shows the analogue of Theorem \ref{thm:torsionAmbro} in this context.

\section{Parabolic fibrations}

Finally, in this section we discuss a more general situation than that of lc-trivial fibrations.

Let $g\colon (X,\Delta)\to Z$ be a surjective morphism, where $(X,\Delta)$ is a klt projective pair and $Z$ is a projective variety. Assume that $g_*\OO_X\big(m(K_X+\Delta)\big)\neq0$ for some positive integer $m$, and consider the relative Iitaka fibration $f\colon X\dashrightarrow Y$ associated to $K_X+\Delta$. Possibly by blowing up further, one may assume that $(X,\Delta)$ is log smooth and that $f$ is a morphism. Then if $F$ is a general fibre of $f$, we have $\kappa\big(F,(K_X+\Delta)|_F\big)=0$, however $K_X+\Delta$ is not necessarily a pullback from $Y$. One still wonders if there is a canonical bundle formula for the map $f$.

The resulting formula is the canonical bundle formula of Fujino and Mori \cite{FM00}. We first need a definition, which is justified from the setup above.

\begin{dfn}
A \emph{parabolic fibration} is a fibration $f\colon (X,\Delta)\to Y$, where $(X,\Delta)$ is a projective klt pair, $Y$ is a smooth projective variety and if $F$ is the generic fibre of $f$, then $\kappa\big(F,(K_X+\Delta)|_F\big)=0$.
\end{dfn}

The following is \cite[Section 4]{FM00}, the \emph{canonical bundle formula} of Fujino and Mori.

\begin{thm}\label{thm:FM}
Let $f\colon (X,\Delta)\to Y$ be a parabolic fibration, where $\Delta$ is effective. Then there is a commutative diagram
$$
\xymatrix{
X' \ar[r]^{\tau'} \ar[d]_{f'} & X\ar[d]^{f}\\
Y' \ar[r]_{\tau}&Y,
}
$$
where $\tau$ and $\tau'$ are birational, $X'$ and $Y'$ are smooth, and $f'$ has connected fibres, such that the following holds.

There exist effective $\Q$-divisors $B^+$ and $B^-$ on $X'$ without common components, a $\Q$-divisor $\Delta'\geq0$ on $X'$ and $\Q$-divisors $B_{Y'}$ and $M_{Y'}$ on $Y$ such that
$$K_{X'}+\Delta'+B^-\sim_\Q f'^*(K_{Y'}+B_{Y'}+M_{Y'})+B^+,$$
with the following properties:
\begin{enumerate}
\item[(i)] the pair $(X',\Delta')$ is klt and log smooth, and there exists an effective exceptional divisor $E$ on $X'$ such that
$$K_{X'}+\Delta'\sim_\Q\tau'^*(K_X+\Delta)+E,$$ 
\item[(ii)] $f'_*\OO_{X'}(\lfloor nB^+\rfloor)\simeq\OO_{Y'}$ for all $n\in\N$,
\item[(iii)] $B^-$ is $f'$-exceptional and $\tau'$-exceptional,
\item[(iv)] the induced map $f'\colon(X',\Delta'+B^--B^+)\to Y'$ is a klt-trivial fibration, and $B_{Y'}$ and $M_{Y'}$ are the corresponding discriminant and moduli divisors,
\item[(v)] the pair $(Y',B_{Y'})$ is klt, $B_{Y'}\geq0$ and $M_{Y'}$ is nef,
\item[(vi)] for every $n\in\N$ sufficiently divisible we have 
$$H^0\big(X,n(K_X+\Delta)\big)\simeq H^0\big(X',n(K_{X'}+\Delta')\big)\simeq H^0\big(Y',n(K_{Y'}+B_{Y'}+M_{Y'})\big).$$
\end{enumerate}
\end{thm}
 
There are several non-trivial parts of this formula which do not follow from considerations in the previous sections: the existence of divisors $B^+$ and $B^-$ with the properties (ii) and (iii) above, as well as the fact that $B_{Y'}$ is effective.

\medskip

Part (vi) follows immediately from (i), (ii) and (iii). We sketch how (iv) follows from (i), (ii) and (iii), following \cite[Lemma 4.2]{Amb04a}. Let $F'$ be a general fibre of $f'$, and we define the divisor $\Xi'$ with respect to $f'\colon(X',\Delta'+B^--B^+)\to Y'$ as in Definition \ref{dfn:lctrivial}(d). We may assume that $\Xi'=-\Delta'-B^-+B^+$, and $B^-|_{F'}=0$ by (iii).

We have $(K_{X'}+\Delta'+B^--B^+)|_{F'}\sim_\Q0$ by construction and $\kappa\big(F',(K_{X'}+\Delta')|_{F'}\big)=0$ by (i), hence 
$$\kappa(F',B^+|_{F'})=\kappa\big(F',(B^+-B^-)|_{F'}\big)=0.$$
Since there exists a positive integer $b$ such that $\lceil B^+\rceil |_{F'}\leq b B^+|_{F'}$, this implies $\kappa\big(F',\lceil B^+\rceil|_{F'}\big)=0$. Therefore,
\begin{align*}
\kappa\big(F',\lceil\Xi'\rceil |_{F'}\big)&=\kappa\big(F',\lceil-\Delta'-B^-+B^+\rceil |_{F'}\big)\\
&\leq\kappa\big(F',\lceil-\Delta'\rceil |_{F'}+\lceil-B^-\rceil |_{F'}+\lceil B^+\rceil |_{F'}\big)\\
&\leq \kappa\big(F',\lceil B^+\rceil |_{F'}\big)=0.
\end{align*}
This shows part (d) of Definition \ref{dfn:lctrivial}, and the rest is easy.

\medskip

In order to apply the previous result, we recall that for a log canonical pair $(X,\Delta)$, the ring
$$R(X,K_X+\Delta)=\bigoplus_{n\in\N}H^0\big(X,\lfloor n(K_X+\Delta)\rfloor\big)$$
is the \emph{canonical ring} of $(X,\Delta)$. We also recall that for a graded ring $R=\bigoplus\limits_{n\in\N}R_n$, the \emph{$d$-th Veronese subring} of $R$ is defined as $R^{(d)}:=\bigoplus\limits_{n\in\N}R_{dn}$.

An immediate consequence of Theorem \ref{thm:FM} is the following result \cite[Theorem 5.2]{FM00}, which has widespread use in the Minimal Model Program. It often allows to pass from a pair $(X,\Delta)$ with $\kappa(X,K_X+\Delta)\geq0$ to a pair $(X',\Delta')$ on which $K_{X'}+\Delta'$ is big.

\begin{thm}
Let $(X,\Delta)$ be a projective klt pair with $\kappa(K_X+\Delta)=\ell\geq0$. Then there exist an $\ell$-dimensional klt pair $(X',\Delta')$ with $\kappa(X',K_{X'}+\Delta')=\ell$ and positive integers $d$ and $d'$ such that
$$R(X, K_X + \Delta)^{(d)} \simeq R(X', K_{X'}+\Delta')^{(d')}.$$
\end{thm}

The proof follows immediately from Theorem \ref{thm:FM}(vi), by combining it with the proof of Theorem \ref{thm:kltdescent}; see also the proof of Theorem \ref{thm:cbf} below.

Assume now that $(X,\Delta)$ has simple normal crossings and let $\Delta^+$ and $\Delta^-$ be effective divisors without common components such that $\Delta=\Delta^+-\Delta^-$. Then if $\kappa\big(F,(K_X+\Delta^+)|_F\big)=0$ for a general fibre $F$ of $f$, and if there exists a good model of $(F,\Delta^+|_F)$, then it was shown in \cite[Theorem 3.13]{Fuj15} that the moduli b-divisor is b-nef and good in the sense of \cite[Definition 3.2]{Amb05a}; this is an application of MMP techniques from \cite{FG14} and \cite[Theorem 3.3]{Amb05a}. 

\medskip

We finish the paper with the following result, which can sometimes be used in order to avoid running a Minimal Model Program; for instance, compare the proofs of \cite[Lemma 4.4]{GL13} and \cite[Theorem 5.3]{LP18a}. Note that $\nu(X,L)$ denotes the \emph{numerical dimension} of a divisor $L$ on a projective variety $X$, see for instance \cite[\S2.2]{LP18a} for basic properties and related references.

\begin{thm}\label{thm:cbf}
Let $(X,\Delta)$ be a projective klt pair and let $f\colon (X,\Delta)\to Y$ be a parabolic fibration such that $\nu\big(F,(K_X+\Delta)|_F\big)=0$ for a general fibre $F$ of $f$. Then there exists a commutative diagram
\[
\xymatrix{ 
X' \ar[r]^{\pi'} \ar[d]_{f'} & X \ar[d]^{f}\\
Y'\ar[r]_{\pi} & Y,
}
\]
where $X'$ and $Y'$ are smooth, $f'$ has connected fibres, $\pi$ and $\pi'$ are birational, and such that, if we write
$$K_{X'}+\Delta'\sim_\Q \pi'^*(K_X+\Delta)+E',$$
where $\Delta'$ and $E'$ have no common components, then:
\begin{enumerate}
\item[(i)] we have
$$K_{X'}+\Delta'+B^-\sim_\Q f'^*(K_{Y'}+\Delta_{Y'})+B^+,$$
where the pair $(Y',\Delta_{Y'})$ is klt, and the divisors $B^+$ and $B^-$ are effective and have no common components,
\item[(ii)] $B^-$ is $\pi'$-exceptional and $f'$-exceptional,
\item[(iii)] we have $f'_*\OO_{X'}(\lfloor\ell B^+\rfloor)\simeq\OO_{Y'}$ for all positive integers $\ell$.
\end{enumerate}
Moreover, if $K_X+\Delta$ is pseudoeffective, then $K_{Y'}+\Delta_{Y'}$ is pseudoeffective.
\end{thm}

\begin{proof}
By \cite[Corollaire 3.4]{Dru11} and \cite[Corollary V.4.9]{Nak04} we have
\begin{equation}\label{eq:1}
\kappa\big(F,(K_X+\Delta)|_F\big)=\nu\big(F,(K_X+\Delta)|_F\big)=0.
\end{equation} 
By Theorem \ref{thm:FM} there exists a diagram as in the theorem such that (ii) and (iii) hold, as well as
$$K_{X'}+\Delta'+B^-\sim_\Q f'^*(K_{Y'}+B_{Y'}+M_{Y'})+B^+,$$
where $(Y',B_{Y'})$ is klt and $M_{Y'}$ is the moduli part of the associated klt-trivial fibration $f'\colon (X',\Delta'+B^--B^+)\to Y'$. Then analogously as in the proof of Theorem \ref{thm:kltdescent} one shows that there exists an effective $\Q$-divisor $\Delta_{Y'}\sim_\Q B_{Y'}+M_{Y'}$ such that the pair $(Y',\Delta_{Y'})$ is klt, which gives (i).

Finally, if $F'$ is a general fibre of $f'$, we have $\nu\big(F',(K_{X'}+\Delta')|_{F'}\big)=0$ by \eqref{eq:1} and by \cite[Lemma 2.3]{LP18a}. Therefore, there exists a good model of $(F',\Delta'|_{F'})$ by \cite[Corollaire 3.4]{Dru11} and \cite[Corollary V.4.9]{Nak04}, hence $(K_{X'}+\Delta')|_{F'}$ is geometrically abundant in the sense of \cite[Definition V.2.23]{Nak04}. Then by \cite[Lemma V.2.27]{Nak04} for an ample divisor $A$ on $Y'$ and for any positive rational number $\varepsilon$, the divisor $K_{X'}+\Delta'+\varepsilon f'^*A$ is geometrically abundant. In particular, $\kappa(X',K_{X'}+\Delta'+\varepsilon f'^*A)\geq0$, and hence by (i) and by (iii) we have
\begin{align*}
\kappa\big(Y',K_{Y'}+\Delta_{Y'}+\varepsilon A)&=\kappa\big(X',f'^*(K_{Y'}+\Delta_{Y'}+\varepsilon A)+B^+\big)\\
&=\kappa(X',K_{X'}+\Delta'+B^-+\varepsilon f'^*A)\geq0.
\end{align*}
Since this holds for any positive rational number $\varepsilon$, we conclude that $K_{Y'}+\Delta_{Y'}$ is pseudoeffective, as desired.
\end{proof}

\bibliographystyle{amsalpha}

\bibliography{biblio}

\end{document}